\documentclass[12pt,a4paper]{article}
\usepackage{a4}

\usepackage{mathrsfs} %mathscr
\usepackage{stmaryrd}
\usepackage{amssymb} %Box
\usepackage{amsmath} %Operator
\usepackage{amsthm} %proof
\usepackage{xcolor}
  \usepackage[pdftex]{graphicx}

\newtheorem{theorem}{Theorem}
\newtheorem{proposition}[theorem]{Proposition}
\newtheorem{lemma}[theorem]{Lemma}

\newtheorem{claim}[theorem]{Claim}

\newcommand{\distind}[1]{\chi'_{#1}}
\newcommand{\distmat}[1]{\mu_{#1}}

\newcommand{\G}[2]{G_{#1,#2}}

\newcommand{\eps}{\varepsilon}

\newcommand{\neighbour}{\hat{N}}
\newcommand{\spanning}{\hat{S}}
\newcommand{\distant}[1]{A_{#1}}
\newcommand{\distantA}{A}
\newcommand{\distantB}{B_t}
\newcommand{\betw}[3]{[#1,#2]_#3}
\newcommand{\eind}[2]{#1\llbracket #2\rrbracket}

\title{The distance-$t$ chromatic index of graphs}

\author{
 {Tom\'a\v{s} Kaiser} \thanks{Supported by project P202/12/G061 of the Czech Science Foundation.} \\
Department of Mathematics, \\
Institute for Theoretical Computer Science (CE--ITI), and\\
European Centre of Excellence NTIS \\
(New Technologies for the Information Society) \\
 University of West Bohemia\\ %, Univerzitn\'i 8, 306 14 \\
 Plze\v{n}, Czech Republic \\
 {\tt kaisert@kma.zcu.cz}
  \and 
 {Ross J. Kang} \thanks{Supported by a NWO Veni Grant. Part of this work was carried out while this author was at Durham University, supported by EPSRC, grant EP/G066604/1.} \\
 Mathematical Institute \\
 Utrecht University \\
 Utrecht, Netherlands \\
 {\tt ross.kang@gmail.com}
}

\begin{document}
\maketitle

\begin{abstract}
We consider two graph colouring problems in which edges at distance at most $t$ are given distinct colours, for some fixed positive integer $t$.  We obtain two upper bounds for the distance-$t$ chromatic index, the least number of colours necessary for such a colouring. One is a bound of $(2-\eps)\Delta^t$ for graphs of maximum degree at most $\Delta$, where $\eps$ is some absolute positive constant independent of $t$.  The other is a bound of $O(\Delta^t/\log \Delta)$ (as $\Delta\to\infty$) for graphs of maximum degree at most $\Delta$ and girth at least $2t+1$.  The first bound is an analogue of Molloy and Reed's bound on the strong chromatic index. The second bound is tight up to a constant multiplicative factor, as certified by a class of graphs of girth at least $g$, for every fixed $g \ge 3$, of arbitrarily large maximum degree $\Delta$, with distance-$t$ chromatic index at least $\Omega(\Delta^t/\log \Delta)$.

  
\end{abstract}

% \begin{keyword}
%   graph colouring \sep distance edge-colouring \sep strong chromatic index \sep graph powers
% \end{keyword}

%%%%%%%%%%%%%%%%%%%%%%%%%%%%%%%%%%%%%%%%%%%%%%%%%%%%%%%%%%%%%%%%%%%%%%

\section{Introduction}\label{sec:intro}

Given a graph $G = (V, E)$ and a positive integer $t$, a {\em distance-$t$ edge-colouring} of $G$ is a colouring of the edges such that no two edges within distance $t$ are given the same colour. Here, the distance between two edges is defined as the number of {\em vertices} in a shortest path between them.  Adjacent edges have distance $1$ and thus a distance-$1$ edge-colouring is just a proper edge-colouring.  The {\em distance-$t$ chromatic index} of $G$, denoted $\distind{t}(G)$, is the least integer $k$ such that there exists a distance-$t$ edge-colouring of $G$ using $k$ colours. Distance-$2$ edge-colourings, also known as {\em strong edge-colourings}, have a rich history, going back to problems posed in 1985 (cf.~\cite{FGST89}).

The distance-$t$ edge-colouring problem is related to colouring powers of graphs. Observe that $\distind{t}(G) = \chi((L(G))^t)$, where $\chi(\cdot)$ denotes the chromatic number, $L(\cdot)$ denotes the line graph, and the $t$-th power of a graph is the graph obtained by adding the edges between pairs of vertices at distance at most $t$. The distance between two vertices is defined as the number of {\em edges} in a shortest path between them.  To be unambiguous, we remark here that, for us, the distance between a vertex and an edge is the smaller of the distances between the vertex and each endvertex of the edge.

In this paper, we study the distance-$t$ chromatic index of graphs of bounded maximum degree.  In this setting, observe first that, if $G$ has maximum degree at most $\Delta$, then easily
\begin{align*}
  \distind{t}(G) \le 1+ 2 \sum_{j=1}^{t}(\Delta - 1)^{j} < 2\Delta^t,
\end{align*}
since the maximum degree of $(L(G))^t$ is at most $2 \sum_{j=1}^{t}(\Delta - 1)^{j}$.
It is natural to wonder whether this trivial upper bound can be improved upon.
Erd\H{o}s and Ne\v{s}et\v{r}il asked this in 1985 (cf.~\cite{FGST89}) for the $t = 2$ case, and in particular asked for some constant $\eps > 0$ such that for all graphs $G$ of maximum degree at
most $\Delta$, $\distind{2}(G) \le (2-\eps)\Delta^2$.
They pointed out a simple class of graphs (blown-up $5$-cycles)
of arbitrarily large maximum degree $\Delta$ with strong chromatic
index $5\Delta^2/4$.
For $t \ge 3$, the second author together with Putra Manggala~\cite{KaMa11+} recently observed that there are examples (in particular, some specific Hamming graphs) that are regular of arbitrarily large degree $\Delta$ and have distance-$t$ chromatic index greater than $\Delta^t/\left(2(t-1)^{t-1}\right) = \Omega(\Delta^t)$, thus certifying that the trivial upper bound cannot be replaced by any bound that is $o(\Delta^t)$.

In the case $t=3$, an alternative construction noted in~\cite{KaMa11+} based on the projective plane yields bipartite regular graphs of girth $6$ with arbitrarily large degree $\Delta$ and distance-$3$ chromatic index $\Delta^3-\Delta^2+\Delta$.
It is worth remarking here that higher-dimensional projective geometries, the so-called {\em generalised polygons} also yield constructions in the cases $t=4$ and $t=6$.  In particular, using generalised quadrangles (respectively, hexagons) it can be checked that there are bipartite regular graphs $G$ of girth $8$ (respectively, $12$) with arbitrarily large degree $\Delta$ such that $\distind{4}(G) = (1 + o(1))\Delta^4$ (respectively, $\distind{6}(G) = (1 + o(1))\Delta^6$).  Although there exist generalised octagons, these unfortunately do not have the appropriate incidence properties for good bounds on $\distind{8}$; moreover, Feit and Higman~\cite{FeHi64} showed that finite generalised $n$-gons only exist for $n\in\{2,3,4,6,8\}$.

More than a decade after it was posed, Molloy and Reed~\cite{MoRe97}, using a combination of structural and probabilistic methods, were able to affirm the above strong chromatic index question of Erd\H{o}s and Ne\v{s}et\v{r}il, with $\eps$ taken to be some constant greater than $0.002$.
Our first result provides a distance-$t$ ($t \ge 3$) analogue of this bound with an absolute choice of $\eps$ valid for all $t$.

\begin{theorem}\label{thm:main1}
  Let $\eps = 0.00008$ and let $t \ge 2$ be an integer. For large
  enough $\Delta_0$, the distance-$t$ chromatic index $\distind{t}(G)$
  of any graph $G$ of maximum degree $\Delta \ge \Delta_0$ is at most
  $(2-\eps)\Delta^t$.
\end{theorem}

\noindent
We point out that our proof of Theorem~\ref{thm:main1} provides a different solution to the problem of Erd\H{o}s and Ne\v{s}et\v{r}il (albeit with a weaker constant), by a method that applies uniformly to all $t\ge2$.  We have found that the specific argument of Molloy and Reed can be adapted to the case $t=3$ (with a choice of $\eps > 0.0002$), but omit the full details since it remains unclear if this can be extended any further to, say, $t=4$ or $t=5$.

Our second result is a $o(\Delta^t)$ upper bound for graphs of large girth.  The following may be considered as an extension of an $O(\Delta^2/\log \Delta)$ bound on $\distind{2}$ for $C_4$-free graphs due to Mahdian~\cite{Mah00}.

\begin{theorem}\label{thm:main2}
  Let $t \ge 2$ be an integer.  For all graphs $G$ of girth at
  least $2t+1$ and maximum degree at most $\Delta$ it holds that
  $\distind{t}(G) = O(\Delta^t/\log \Delta)$.
\end{theorem}

\noindent
By a probabilistic construction, this bound is tight up to a constant
factor dependent upon $t$.

\begin{proposition}\label{prop:main2}
  There is a function $f = f(\Delta,t) = (1+o(1))\Delta^t/(t\log
  \Delta)$ (as $\Delta\to \infty$) such that, for every $g\ge 3$ and
  every $\Delta$, there is a graph $G$ of girth at least $g$ and
  maximum degree at most $\Delta$ with $\distind{t}(G)\ge
  f(\Delta,t)$.
\end{proposition}

\noindent
The complete bipartite graphs $K_{\Delta,\Delta}$ are regular of arbitrarily large degree $\Delta$, have girth $4$ and strong chromatic index $\Delta^2$ as noted in~\cite{FGST89}. Similarly, the examples based on projective geometries mentioned above are regular of arbitrarily large degree $\Delta$, have girth $2t$ and distance-$t$ chromatic index $(1+o(1))\Delta^t$ for $t\in \{3,4,6\}$. Thus the girth condition in Theorem~\ref{thm:main2} cannot be weakened for the cases $t\in\{2,3,4,6\}$; however, for $t=5$ or $t\ge7$, no construction is known of graphs of arbitrarily large maximum degree $\Delta$, with girth greater than $3$ and $\Omega(\Delta^t)$ distance-$t$ chromatic index.

We note that, if the girth condition in Theorem~\ref{thm:main2} is strengthened, then it is possible to improve the bound to one that is tight up to a constant multiple independent of $t$.
For the chromatic number of the $t$-th power of a graph of maximum degree $\Delta$ and girth at least $g$,
Alon and Mohar~\cite{AlMo02} outlined an upper bound of $O(\Delta^t/(t \log \Delta))$ if $g \ge 3t+1$.  In a similar fashion, the following can be shown to hold, and we leave the details to the reader.
\begin{proposition}\label{thm:girthbigger}
  Let $t \ge 2$ be an integer.  For all graphs $G$ of girth at
  least $3t-2$ and maximum degree at most $\Delta$ it holds that
  $\distind{t}(G) = O(\Delta^t/(t \log \Delta))$.
\end{proposition}

Our upper bounds (Theorems~\ref{thm:main1} and~\ref{thm:main2} and Proposition~\ref{thm:girthbigger}) follow by first performing a structural
neighbourhood count, then applying a known chromatic number upper
bound given each of the neighbourhood subgraphs has bounded
sparsity.  We briefly review chromatic number bounds for graphs with
sparse neighbourhoods in Section~\ref{sec:sparse}; these tools
all depend on the Lov\'asz Local Lemma.

The structure of the paper is as follows.  In
Section~\ref{sec:sparse}, we review known results on the colouring of
graphs with sparse neighbourhoods.  In Section~\ref{sec:MoRe}, we
develop the general upper bound. In Section~\ref{sec:Mah1}, we prove
the upper bound for graphs of high girth and also give the
probabilistic construction showing said bound is tight up to a
constant multiple.  Some open problems are mentioned in the
conclusion.

\subsection*{Notation and preliminaries}
Graphs in this paper have no parallel edges or loops. Let $G = (V,E)$
be a graph. We use the following notation.
\begin{itemize}
\item For $u,v\in V$, a \emph{$uv$-walk} is a sequence
  $(u,e_1,v_1,\dots,e_\ell,v)$ such that each element in the sequence
  is incident (in $G$) with the following element in the sequence.
  The \emph{length} of such a walk is $\ell$. We may also refer to the
  above as an $e_1e_\ell$-walk and instead write
  $(e_1,v_1,\dots,v_{\ell-1},e_\ell)$ if no confusion between the
  vertices and edges can arise.
\item For $A,B\subseteq V$, the symbol $\betw A B G$ denotes the
  set of all edges of $E$ that have one endvertex in $A$ and one in
  $B$.
\item For $A\subseteq V$, $G[A]$ denotes the subgraph of $G$ induced
  by the vertices of $A$.
\item For $X\subseteq E$, $\eind G X$ denotes the graph $(V,X)$.
\item For $v\in V$, the \emph{neighbourhood} $N_G(v)$ of $v$ is the
  set of all vertices that are adjacent to $v$, and the degree
  $\deg_G(v)$ of $v$ is the cardinality of $N_G(v)$. We may omit the
  subscript $_G$ when the context is clear.
\item For $v\in V$, the set of edges of $G$ incident with $v$ is
  denoted by $G(v)$.
\end{itemize}

%%%%%%%%%%%%%%%%%%%%%%%%%%%%%%%%%%%%%%%%%%%%%%%%%%%%%%%%%%%%%%%%%%%%%%

\section{Colouring graphs with sparse neighbourhoods}\label{sec:sparse}

It is easy to see that a graph's chromatic number is at most its
maximum degree plus one.  Beginning with the unpublished work of
Johansson~\cite{Joh96} in the mid-1990's, there were several results
improving upon this bound in the case of graphs for which the number
of edges spanning each neighbourhood set is bounded.  We state some of
these bounds here.  All of these were proved using the probabilistic
method, with the Lov\'asz Local Lemma.

\begin{theorem}[Johansson~\cite{Joh96}]\label{thm:Johsparse}
  For all triangle-free graphs $\hat{G}$ with maximum degree at most
  $\hat{\Delta}$, it holds that $\chi(\hat{G}) = O(\hat{\Delta}/\log
  \hat{\Delta})$.
\end{theorem}

\begin{theorem}[Molloy and Reed~\cite{MoRe97}]\label{thm:MoResparse}
  Let $\delta, \eps>0$ be such that $\eps <
  \frac{\delta}{2(1-\eps)}e^{-\frac{3}{1-\eps}}$. There is some large
  enough $\hat{\Delta}_0$ such that the following holds.  If
  $\hat{G}=(\hat{V},\hat{E})$ is a graph with maximum degree at most
  $\hat{\Delta}\ge \hat{\Delta}_0$ such that for each $\hat{v}\in
  \hat{V}$ there are at most $(1-\delta)\binom{\hat{\Delta}}2$ edges
  spanning $N(\hat{v})$, then $\chi(\hat{G})\le(1-\eps)\hat{\Delta}$.
\end{theorem}

\begin{theorem}[Alon, Krivelevich and
  Sudakov~\cite{AKS99}]\label{thm:AKSsparse}
  For all graphs $\hat{G} = (\hat{V},\hat{E})$ with maximum degree at
  most $\hat{\Delta}$ such that for each $\hat{v}\in \hat{V}$ there
  are at most $\hat{\Delta}^2/f$ edges spanning $N(\hat{v})$, it holds
  that $\chi(\hat{G}) = O(\hat{\Delta}/\log f)$.
\end{theorem}

\begin{theorem}[Mahdian~\cite{Mah00}]\label{thm:Mahsparse}
  Let $\eps>0$.  There is some large enough $\hat{\Delta}_0$ such that
  the following holds.  If $\hat{G}=(\hat{V},\hat{E})$ is a graph with
  maximum degree at most $\hat{\Delta}\ge \hat{\Delta}_0$ such that
  for each $\hat{v}\in \hat{V}$,
  \begin{enumerate}
  \item the largest independent set in the subgraph spanning
    $N(\hat{v})$ has at most $O(\hat{\Delta}^{1/2})$ vertices, and
  \item with the exception of at most $O(\hat{\Delta}^{1/2})$
    neighbours of $\hat{v}$, every vertex $\hat{u}\in \hat{V}$ has at
    most $O(\hat{\Delta}^{1/2})$ common neighbours with $\hat{v}$,
  \end{enumerate}
  then $\chi(\hat{G})\le(1+\eps)\hat{\Delta}/\log \hat{\Delta}$.
\end{theorem}

%%%%%%%%%%%%%%%%%%%%%%%%%%%%%%%%%%%%%%%%%%%%%%%%%%%%%%%%%%%%%%%%%%%%%%

\section{A general upper bound}\label{sec:MoRe}

In this section, we prove Theorem~\ref{thm:main1}. We use
Theorem~\ref{thm:MoResparse} in combination with the following
structural lemma.

\begin{lemma}\label{lem:main}
  Let $t \ge 2$ be an integer. For $\delta = 1/618 \approx
  0.001618$ and large enough $\Delta_0$, it holds that, if
  $\hat{G}=(\hat{V},\hat{E})$ is $L(G)^t$ for some graph $G = (V,E)$
  of maximum degree at most $\Delta \ge \Delta_0$, then there are at
  most $(2-2\delta)\cdot\Delta^{2t}$ edges spanning the neighbourhood
  of each vertex of $\hat{V}$.
\end{lemma}

\begin{proof}
  Without loss of generality, we assume that $G = (V,E)$ is
  $\Delta$-regular.  Let $\hat{G}=(\hat{V},\hat{E})$ be $L(G)^t$ and
  pick an arbitrary $e^*\in\hat V$. Note that $e^*$ is an edge of $G$.

  Set
  \begin{align*}
    \neighbour = N_{\hat{G}}(e^*) \quad \text{ and } \quad \spanning =
    E(\hat{G}[\neighbour])
  \end{align*}
  and observe that $|\neighbour| < 2\Delta^t$. The proof of the
  lemma is by contradiction. Suppose that
  \begin{align}\label{eqn:spanning}
    |\spanning| > (2-2\delta)\Delta^{2t}.
  \end{align}

  For $X,Y\subseteq E$, let $W_\ell(X,Y)$ denote the collection of
  walks in $G$ of length at most $\ell$ starting with an edge in $X$
  and ending with an edge in $Y$, and let $w_\ell(X,Y) =
  |W_\ell(X,Y)|$. If $X$ or $Y$ contains only a single edge, then we
  omit the set brackets and write, for example,~$w_\ell(e,f)$.  For $e,f\in
  E$, let
  \begin{equation*}
    \tau(e,f) = \max\{w_{t+1}(e,f)-1,0\}.
\end{equation*}
Observe that, for $e,f\in \hat{V}$, we have $w_{t+1}(e,f) > 0$ if and
only if $ef\in \hat{E}$ or $e=f$.  Thus
  \begin{align}\label{eq:sum-tau}
    \sum_{e,f\in\neighbour}\tau(e,f)
    &= \sum_{e,f\in\neighbour}w_{t+1}(e,f) - 
    (2|\spanning| + |\neighbour|)\notag\\
    &< w_{t+1}(\neighbour,\neighbour) - 2|\spanning|.
  \end{align} 
  Combining \eqref{eq:sum-tau} with the observation that
  $w_{t+1}(\neighbour,E) \leq |\neighbour| \cdot 2\Delta^t < 4\Delta^{2t}$, we obtain
  \begin{align*}
    4\Delta^{2t} & > w_{t+1}(\neighbour,E)
    = w_{t+1}(\neighbour,\neighbour) + w_{t+1}(\neighbour,E-\neighbour)\\
    & = \sum_{e,f\in\neighbour}\tau(e,f) + 2|\spanning| +
    w_{t+1}(\neighbour,E-\neighbour).
  \end{align*}
  Thus with~\eqref{eqn:spanning} we have
  \begin{align}\label{eqn:maina}
    \sum_{e,f\in\neighbour}\tau(e,f) +
    w_{t+1}(\neighbour,E-\neighbour) < 4\delta\cdot \Delta^{2t}.
  \end{align}

  For $i\in\{0,\dots,t\}$, let $\distant{i}$ be the set of vertices of
  $G$ at distance $i$ from $e^*$. In particular, $\distant{0}$
  consists of the two endvertices of $e^*$. Set $\distantA =
  \distant{0} \cup \dots \cup \distant{t-1}$. Observe that $|A_i| \leq
  2(\Delta-1)^i$ for all $i\in\{0,\dots,t\}$, and that $|A| <
  2\Delta^{t-1}$.

  Furthermore, let $\distantB$ be the set of vertices
  $u\in\distant{t}$ such that
  \begin{equation}
    \label{eq:Delta-half}
    \deg_{\eind G\neighbour}(u) \geq \Delta/2.
  \end{equation}
  Trivially, \eqref{eq:Delta-half} also holds for any $u\in A$.

  For $u,v \in V$, let us write $\sigma_\ell(u,v)$ for the number of
  $uv$-walks of length at most $\ell$ whose first edge is in
  $\neighbour$.  Setting $\alpha = 1-5\delta$, $\beta = \alpha/36$ and
  $\gamma = \beta/(2\beta+1)$, the desired contradiction will be
  obtained with use of the following two claims.

  \begin{claim}\label{clm:maina}
    \begin{align*}
      \sum_{\substack{u,v\in \distantA\cup\distantB \\ \sigma_t(u,v) \ge
        \beta\Delta}}\sigma_t(u,v) > \alpha \cdot \Delta^{2t-1}.
    \end{align*}
  \end{claim}

  \begin{claim}\label{clm:mainb}
    If $u,v \in A\cup B_t$ are such that $\sigma_t(u,v) \ge \beta\Delta$, then
    \begin{align*}
      \sum_{\substack{e\in \eind G \neighbour(u) \\ f\in \eind G \neighbour(v)}}
      \tau(e,f) \ge \gamma\Delta \cdot \sigma_t(u,v).
    \end{align*}
  \end{claim}

  Before proving these claims, let us see how they imply the lemma.
  We sum the inequality in Claim~\ref{clm:mainb} over all $u,v\in
  \distantA\cup\distantB$ such that $\sigma_t(u,v) \ge \beta\Delta$.
  Correcting the possible overcounting by a factor of $4$ (which corresponds to the labellings/orderings of $e,f$), we obtain
  \begin{align*}
    \sum_{e,f\in\neighbour} \tau(e,f) \ge \frac{\gamma\Delta}4 \cdot 
    \sum_{\substack{u,v\in \distantA\cup\distantB \\ \sigma_t(u,v) \ge
      \beta\Delta}} \sigma_t(u,v) > \frac{\gamma\Delta}4 \alpha \cdot
    \Delta^{2t-1} = \frac{\gamma\alpha}4 \cdot \Delta^{2t},
  \end{align*}
  where the last inequality follows from Claim~\ref{clm:maina}.  This
  contradicts~\eqref{eqn:maina} if $\gamma\alpha/4 > 4\delta$.
  Substituting the values of $\alpha$, $\beta$ and $\gamma$, we find,
  after some routine manipulation, that $\gamma\alpha/4 > 4\delta$ is
  equivalent to $185\delta^2-618\delta+1>0$, which is true if $0 <
  \delta \le 1/618$.  As $e^*$ is an arbitrary element of $\hat{V}$,
  the resultant contradiction proves the lemma.

  \begin{proof}[Proof of Claim~\ref{clm:maina}]
    This proof amounts to verifying the following inequality:
    \begin{align}\label{eqn:clmmaina}
      \sum_{u,v\in \distantA\cup\distantB}\sigma_t(u,v) >
      2\alpha\cdot \Delta^{2t-1}.
    \end{align}
    The reason for this is as follows. Assuming the truth of
    \eqref{eqn:clmmaina}, 
    the mean of $\sigma_t(u,v)$ over all
    $u,v\in\distantA\cup\distantB$ is at least $2\alpha\cdot \Delta^{2t-1}/|\distantA\cup\distantB|^2$.
    To simplify this lower bound, we recall that $|A|< 2\Delta^{t-1}$;
    furthermore, since the number of edges of $G$ between $A_{t-1}$
    and $A_t$ is at most $|A_{t-1}| \cdot \Delta < 2\Delta^t$, the
    size of $B_t$ is at most $2\Delta^t/(\Delta/2) = 4\Delta^{t-1}$
    by the definition. Hence, assuming~\eqref{eqn:clmmaina}, the mean
    of $\sigma_t(u,v)$ as above is at least
    \begin{equation*}
      \frac{2\alpha\Delta^{2t-1}}{36\Delta^{2t-2}} =
      \frac{\alpha\Delta}{18} = 2\beta\Delta.
    \end{equation*}

    Claim~\ref{clm:maina} is now implied by the following elementary
    inequality: if $x_1,\ldots,x_n$ are quantities with mean at least
    $a$, then
    \begin{equation}\label{eq:mean}
      \sum_{i: x_i \ge a/2}x_i > \frac{an}2.
    \end{equation}
    We apply \eqref{eq:mean} to the numbers $\sigma_t(u,v)$, where $u$
    and $v$ range over $A\cup B_t$, setting $a=2\beta\Delta$. Then the
    left hand side of \eqref{eq:mean} coincides with that of the
    inequality in Claim~\ref{clm:maina}, while the right hand side of
    \eqref{eq:mean} is at least $\alpha\Delta^{2t-1}$
    by~\eqref{eqn:clmmaina}. Claim~\ref{clm:maina} follows. Thus, it
    remains to prove the inequality~\eqref{eqn:clmmaina}.

  Before continuing, let us make a basic observation about the
  cardinality of $\distantA$:
  \begin{align}
    (2-2\delta)\Delta^{t-1} < |\distantA|. 
    \label{eqn:distantA}
  \end{align}
  If this does not hold, then
  $|\neighbour| \le \Delta |\distantA| \le (2-2\delta)\Delta^t$ which
  in turn implies that $|\spanning| \le \binom{|\neighbour|}2 <
  (2-2\delta)\Delta^{2t}$, a contradiction to~\eqref{eqn:spanning}.

\begin{figure}
\begin{center}
\input{badwalk.pspdftex} 
\end{center}
\caption{An illustration of how to obtain a walk in $W_{t+1}(\neighbour,E-\neighbour)$ from a bad walk $W$, in the proof of Claim~\ref{clm:maina}.\label{fig:badwalk}}
\end{figure}

    To prove~\eqref{eqn:clmmaina}, we start by counting all walks in
    $G$ of length at most $t$ whose first vertex, say $u$, is in
    $\distantA$.  By~\eqref{eqn:distantA}, there are at least
    $(2-2\delta)\Delta^{t-1}$ choices for $u$, and so the number of
    such walks is at least $(2-2\delta)\Delta^{2t-1}$.  Of these
    walks, those whose last vertex is outside of
    $\distantA\cup\distantB$ we call {\em bad}.  To any bad walk $W$
    with last vertex $v$, we may append any of the at least $\Delta/2$
    edges incident with $v$ and not contained in $\neighbour$ (since
    $v\notin\distantA\cup\distantB$) to obtain a walk in
    $W_{t+1}(\neighbour,E-\neighbour)$.  
    See Figure~\ref{fig:badwalk} for a depiction.
    Since the obtained walks are
    distinct for distinct choices of $u$, $W$ and the appended edge,
    we deduce from~\eqref{eqn:maina} that the number of bad walks is
    less than
    \begin{align*}
      \frac{4\delta\cdot \Delta^{2t}}{\Delta/2} = 8\delta\cdot \Delta^{2t-1}.
    \end{align*}
    We then conclude that more than $(2-2\delta-8\delta)\Delta^{2t-1}
    = 2\alpha \cdot \Delta^{2t-1}$ walks of length at most $t$ with
    first vertex in $A$ are not bad.  This completes the proof
    of~\eqref{eqn:clmmaina} and hence of the claim.
    \renewcommand{\qedsymbol}{$\Diamond$}\end{proof}

  \begin{proof}[Proof of Claim~\ref{clm:mainb}]
    For $u,v\in V$ and $f\in \eind G \neighbour (v)$, we let
    $\sigma_\ell(u;f,v)$ denote the number of $uv$-walks of length at
    most $\ell$ whose first edge is in $\neighbour$ and whose last
    edge is $f$.  Call $f$ {\em relevant} if $\sigma_t(u;f,v)>0$.  For
    this proof, we consider two cases.

    \paragraph{Case~1.}
    \emph{At least $\gamma\Delta$ edges of $\eind G\neighbour(v)$ are
      relevant.}\medskip

    Let $e\in \eind G \neighbour(u)$ and $f\in \eind G \neighbour(v)$
    with $f$ relevant.  Observe that, if $W$ is a $vu$-walk of length
    at most $t$ whose last edge is $e$, then we may prepend $f$ to
    obtain a walk in $W_{t+1}(f,e)$.  A similar observation holds for
    $uv$-walks of length at most $t$.  We find that
    \begin{align*}
      w_{t+1}(e,f) \ge \sigma_t(v;e,u) + \sigma_t(u;f,v).
    \end{align*}
    Since $f$ is relevant, we infer that $\tau(e,f)\ge\sigma_t(v;e,u)$
    and thus
    \begin{align*}
      \sum_{\substack{e\in \eind G\neighbour(u) \\ f\in \eind G\neighbour(v)}}
      \tau(e,f) \ge \gamma\Delta\cdot \sum_{e\in \eind G\neighbour(u)}
      \sigma(v;e,u) = \gamma\Delta\cdot \sigma_t(u,v).
    \end{align*}

    \paragraph{Case~2.}
    \emph{Fewer than $\gamma\Delta$ edges of $\eind G\neighbour(v)$
      are relevant.}\medskip

    As in the last case, if $e\in \eind G\neighbour(u)$ and $f\in
    \eind G\neighbour(v)$, then $w_{t+1}(e,f) \ge
    \sigma_t(u;f,v)$. Using inequality~\eqref{eq:Delta-half} (valid
    for $u$ since $u\in A\cup B_t$) and the assumption of Case~2, we
    obtain
    \begin{align*}
      \sum_{\substack{e\in \eind G\neighbour(u) \\ f\in \eind G\neighbour(v)}}
      \tau(e,f) & \ge \sum_{e\in \eind G\neighbour(u)}\sum_{\substack{{f\in
          \eind
          G\neighbour(v)} \\ \text{$f$ relevant}}} (w_{t+1}(e,f)-1)\\
      & \ge \frac{\Delta}2 \cdot \sum_{\substack{f\in \eind G\neighbour(v) \\ f\text{ relevant}}} (\sigma_t(u;f,v)-1) \\
      & \ge \frac{\Delta}2 \cdot \sigma_t(u,v)-\frac{\Delta}2\cdot
      \gamma\Delta.
    \end{align*}
    What remains is a routine verification that the resulting
    expression is at least $\gamma\Delta\cdot \sigma_t(u,v)$, which follows
    from the assumption that $\sigma_t(u,v) \ge \beta\Delta$.

    \medskip This concludes the proof of Claim~\ref{clm:mainb}.
    \renewcommand{\qedsymbol}{$\Diamond$}\end{proof} Having proved
  Claims~\ref{clm:maina} and~\ref{clm:mainb}, the proof of
  Lemma~\ref{lem:main} is now complete.
\end{proof}

We proceed to prove Theorem~\ref{thm:main1}. Fix an integer $t \geq
2$, $\Delta$ is large enough and $\hat v$ is a vertex of
$\hat G = L(G)^t$. By Lemma~\ref{lem:main}, there are at most
$(2-2\delta)\cdot\Delta^{2t}$ edges spanning $N(\hat v)$, where
$\delta = 1/618$. To apply Theorem~\ref{thm:MoResparse}, we need to
replace the $\Delta^{2t}$ term by $\binom{\hat\Delta}2$, where
$\hat\Delta = 2\Delta^t$ is an upper bound on the maximum degree of
$\hat G$. Since $(2-2\delta)\cdot\Delta^{2t}$ approaches
$(1-\delta)\binom{\hat\Delta}2$ (and $\Delta$ is large enough), the
number of edges spanning $N(\hat v)$ is at most
\begin{equation*}
  (1-\delta')\cdot \binom{\hat\Delta}2,
\end{equation*}
where $\delta'$ is a little smaller than $\delta$, say $\delta' =
1/619$. Setting $\eps'=1/25000$, it is routine to check that
\begin{equation*}
  \eps' < \frac{\delta'}{2(1-\eps')} \cdot e^{-3/(1-\eps')},
\end{equation*}
making Theorem~\ref{thm:MoResparse} applicable to $\hat G$ with $\delta'$ and $\eps'$. We
conclude that
\begin{equation*}
  \chi(\hat G) \leq (1-\eps')\hat\Delta = (2-2\eps') \cdot \Delta^t,
\end{equation*}
which completes the proof of Theorem~\ref{thm:main1}.

%%%%%%%%%%%%%%%%%%%%%%%%%%%%%%%%%%%%%%%%%%%%%%%%%%%%%%%%%%%%%%%%%%%%%%

\section{Graphs of prescribed girth}\label{sec:Mah1}

In this section, we prove Theorem~\ref{thm:main2} and
Proposition~\ref{prop:main2}.

For Theorem~\ref{thm:main2}, we require the following lemma.

\begin{lemma}\label{lem:girth2t+1}
  Let $t \ge 2$ be an integer.  Suppose
  $\hat{G}=(\hat{V},\hat{E})$ is $L(G)^t$ for some graph $G = (V,E)$
  of girth at least $2t+1$ with maximum degree at most $\Delta$.  So
  $\hat{G}$ has maximum degree at most $\hat{\Delta} = 2\Delta^t$.
  Then, for each $\hat{v}\in \hat{V}$, with the exception of at most
  $2\Delta^{t-1}=O(\hat{\Delta}^{1-1/t})$ neighbours of $\hat{v}$,
  every vertex $\hat{u}\in \hat{V}$ has at most
  $(3t+2)\Delta^{t-1}=O(\hat{\Delta}^{1-1/t})$ common neighbours with
  $\hat{v}$.
\end{lemma}

\begin{proof}
  Let $e^*$ be an arbitrary edge of $G$. As in the proof of
  Lemma~\ref{lem:main}, for $i\in\{0,\dots,t\}$, we let $\distant{i}$
  be the set of vertices of $G$ at distance $i$ from $e^*$.

  We say that an edge in $G$ is {\em heavy} if its distance from $e^*$
  is less than $t$; otherwise, we say it is {\em light}.  
  Using the girth condition, it is
  straightforward to verify the following three claims.  To aid the reader, we have included a depiction of $N_{\hat{G}}(e^*)$ in  Figure~\ref{fig:girth2t+1}.

  \begin{claim}\label{clm:girtha}
    Every vertex in $\distant{t-1}$ is incident with at most one heavy
    edge.
  \end{claim}

  \begin{claim}\label{clm:girthb}
    At most one endvertex of a light edge is in $\distant{t-1}$.
  \end{claim}

  \begin{claim}\label{clm:girthc}
    Every vertex in $\distant{t}$ is adjacent to at most two vertices
    in $\distant{t-1}$.
  \end{claim}
  
\begin{figure}
\begin{center}
\input{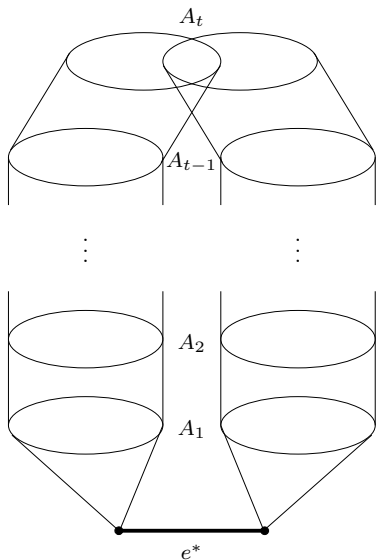} 
\end{center}
\caption{An illustration of $N_{\hat{G}}(e^*)$ in the proof of Lemma~\ref{lem:girth2t+1}.  Each of the induced subgraphs $G[\distant{i}]$, $i \in \{1,\dots t\}$, have no edges.  The light edges of $N_{\hat{G}}(e^*)$ are precisely those between the top two layers.\label{fig:girth2t+1}}
\end{figure}

  Now consider any $f \in N_{\hat{G}}(e^*)$ that is not a heavy edge, i.e.~$f$ is a light edge in $N_{\hat{G}}(e^*)$. We want to estimate the number of members of $N_{\hat{G}}(e^*)$ at distance at most $t$ from $f$ in $G$.  Trivially, there are at most $2\Delta^{t-1}$ heavy edges at distance at most $t$ from $f$ in $G$.  To enumerate the $fe_t$-walks $(f,v_1,e_1,\cdots,v_t,e_t)$ of length $t+1$ where $e_t$ is a light edge in $N_{\hat{G}}(e^*)$, we condition on the least $i\in\{1,\dots,t\}$ such that one of the following events occurs: (i) $v_i\in \distant{t-1}$ and $e_i$ is a heavy edge, or (ii) $v_i\in \distant{t}$ and $e_i$ is a light edge in $N_{\hat{G}}(e^*)$.  By Claim~\ref{clm:girthb}, the set of light edges in $N_{\hat{G}}(e^*)$ is precisely the set $[\distant{t-1},\distant{t}]_G$, and the induced subgraph $G[\distant{t-1}]$ has no edges. Considering the fact that $e_t$ is a light edge in $N_{\hat{G}}(e^*)$, we conclude that the quantity $i$ in the conditioning above is well-defined.

  By Claims~\ref{clm:girtha} and~\ref{clm:girthc}, there are at most three choices for the edge $e_i$, while there are up to $\Delta$ choices for each of the remaining edges in the path.  Therefore, the total number of such paths is upper bounded by $3t\cdot \Delta^{t-1}$. Summing up, with the exception of the at most $2\Delta^{t-1}=O(\Delta^{t-1})$ heavy edges, for every $f \in N_{\hat{G}}(e^*)$ there are at most $(3t+2)\Delta^{t-1} = O(\Delta^{t-1})$ edges at distance at most $t$ from $f$ in $G$.  Since $e^*$ was an arbitrary element of $\hat{V}$, this completes the proof.
\end{proof}

\noindent
The conclusion of Lemma~\ref{lem:girth2t+1} implies that for each $\hat{v} \in \hat{V}$ there are at most $\hat{\Delta}/f$ edges spanning $N(\hat{v})$, for $f = 4\Delta/(3t+4) = O(\hat{\Delta}^{1/t})$. An application of Theorem~\ref{thm:AKSsparse} to $\hat{G}$ then yields Theorem~\ref{thm:main2}.

It is worth noting that, under the conditions of
Lemma~\ref{lem:girth2t+1}, the largest independent set in the subgraph
spanning $N(v)$ has size $O(\hat{\Delta}^{1-1/t})$ for each $v\in
\hat{V}$.  Therefore, if some analogue of Theorem~\ref{thm:Mahsparse}
holds in which the $O(\hat{\Delta}^{1/2})$ conditions are
replaced by $O(\hat{\Delta}^{1-1/t})$ conditions, then we would obtain
a stronger upper bound on $\distind{t}(G)$ for graphs
$G$ of girth $2t+1$.

% \subsection{High-girth probabilistic constructions}\label{sec:Mah2}

Next, to show that the bound in Theorem~\ref{thm:main2} is tight up to
a constant multiple, we use a probabilistic construction
inspired by the classical proof due to Erd\H{o}s~\cite{Erd59} that
there are graphs with arbitrarily large girth and chromatic number.
Similar constructions for strong
edge-colouring and colouring powers of graphs are found in~\cite{AlMo02} and~\cite{Mah00},
respectively.  Let $\G{n}{p}$ be a random graph on $n$ vertices with
edge probability $p$.  We say that a property holds asymptotically
almost surely (a.a.s.)~if it holds with probability tending to $1$ as
$n\to\infty$.  The {\em distance-$t$ matching number $\distmat{t}(G)$}
of a graph $G$ is defined as the size of a largest set of edges in $G$
such that no two of its members are at distance at most $t$.  Clearly,
$\distind{t}(G)\ge |E|/\distmat{t}(G)$ for any graph $G = (V,E)$.  For
the proof of Proposition~\ref{prop:main2}, we require the following
estimate of $\distmat{t}(\G{n}{p})$.

\begin{lemma}[Kang and Manggala~\cite{KaMa11+}]\label{lem:random}
  Let $\eps>0$ and suppose $p = d/n$ with $d \ge d_0$ for some large
  fixed $d_0$.  If
  \[k_{t}=\frac{n}{2 d^{t-1}}\left(t\log d-\log \log d-\log e t+\eps\right),\]
  then $\distmat{t}\left(\G{n}{p}\right) \leq k_{t}$ a.a.s.
\end{lemma}

\begin{proof}[Proof of Proposition~\ref{prop:main2}]
  Let $G = \G{n}{p}$ where $p = d/n$ for $d = (\log n)^{1/(g+1)}$.  By
  Lemma~\ref{lem:random}, $\distmat{t}(G) < n t\log d/(2 d^{t-1})$
  a.a.s.  The following are routine random graph estimates: a.a.s.,
  \begin{enumerate}
  \item the expected number of cycles of length less than $g$ in $G$
    is less than $\log n$;
  \item the number of edges incident with vertices with degree at least
    $d + d/\log d$ is at most $n$; and
  \item the number of edges in $G$ is $(1+o(1))nd/2$.
  \end{enumerate}
  Now remove from $G$ all edges in cycles of length less than $g$ and
  all vertices having degree at least $d+d/\log d$.  By the above
  estimates, the remaining graph a.a.s.~has at least
  $(1+o(1))n\Delta/2$ edges, girth at least $g$, and $\distmat{t} <
  (1+o(1))n t \log \Delta/(2\Delta^{t-1})$.  Thus, a.a.s.~it is a
  graph of girth at least $g$ with distance-$t$ chromatic index at
  least
  \[
  (1+o(1))\frac{n\Delta/2}{n t \log \Delta/(2\Delta^{t-1})} =
  (1+o(1))\frac{\Delta^t}{t\log \Delta}.  \qedhere
  \]
\end{proof}

%%%%%%%%%%%%%%%%%%%%%%%%%%%%%%%%%%%%%%%%%%%%%%%%%%%%%%%%%%%%%%%%%%%%%%

\section{Concluding remarks}\label{sec:conclusion}

It would be interesting to find classes with larger 
distance-$t$ chromatic index with respect to the maximum degree.
Unlike for the case $t=2$ --- Erd\H{o}s and Ne\v{s}et\v{r}il
conjectured that the blown-up $5$-cycle is optimal --- no extremal
classes have been conjectured yet for $t\ge 3$; however, our knowledge
is rather limited. In fact, we know of no classes with
$\Omega(\Delta^t)$ distance-$t$ chromatic index aside from the two
mentioned in the introduction.

We do suspect that the examples from projective geometries are extremal for their respective distance-$t$ chromatic indices, i.e.~$t\in \{3,4,6\}$.  However, then the asymptotic upper bound of $1.99992\Delta^t$ (or slightly lower for $t=3$) would be nearly double the true value.

Further improvements in the constant $\eps$ for the upper bound of
Theorem~\ref{thm:main1} are possible.  In particular, Molloy and
Reed~\cite{MoRe97} alluded to an improvement for
Theorem~\ref{thm:MoResparse}, which if implemented imply immediate
improvements in $\eps$ for Theorem~\ref{thm:main1} as well as for their
own bound on the strong chromatic index.  However, the resulting
bounds could remain far from best possible.  Indeed, there is no class
of examples to preclude an upper bound that replaces the factor
$(2-\eps)$ in Theorem~\ref{thm:main1} by some factor that decreases to
zero as $t\to\infty$.

The cases $t\in\{2,3,4,6\}$ of Theorem~\ref{thm:main2} hint at the
existence for every $t \ge 5$ of some threshold parameter $g'_t$ such that
\begin{itemize}
\item there are graphs of arbitrarily large maximum degree $\Delta$
  with girth $g'_t$ and distance-$t$ chromatic index at least
  $\Omega(\Delta^t)$, and
\item any graph of maximum degree at most $\Delta$ with girth at least
  $g'_t+1$ has distance-$t$ chromatic index at most $O(\Delta^t/\log
  \Delta)$.
\end{itemize}
For $t=5$ and $t\ge7$, it may indeed hold that $g'_t = 2t$, but the graphs in the general (Hamming graph) class of
$\Omega(\Delta^t)$ examples mentioned in the introduction each have many short cycles.  The question if $g'_t$ exists is similar to a question for colouring powers of graphs of
prescribed girth, which was posed over a decade ago~\cite{AlMo02} and remains open.
Alon and Mohar asked for a parameter $g_t$ such that
\begin{itemize}
\item there are graphs of arbitrarily large maximum degree $\Delta$
  with girth $g_t$ and the chromatic number of the $t$-th power of the graph is at least
  $\Omega(\Delta^t)$, and
\item for any graph of maximum degree at most $\Delta$ with girth at least $g_t+1$, the chromatic number of the $t$-th power of the graph is at most $O(\Delta^t/\log
  \Delta)$.
\end{itemize}
They suggested that the ``drastic change'' implied by the existence of such a parameter would occur at around $3t$.
We remark here however that an $O(\Delta^t/\log \Delta)$ bound for
colouring the $t$-th power of graphs of maximum degree $\Delta$ and
girth at least $2t+3$ follows by adapting the method used in
Theorem~\ref{thm:main2}, showing we have (if it exists) $g_t \le 2t+2$ as opposed to $g_t$ being around $3t$.

%%%%%%%%%%%%%%%%%%%%%%%%%%%%%%%%%%%%%%%%%%%%%%%%%%%%%%%%%%%%%%%%%%%%%%

\bibliographystyle{abbrv}
\bibliography{distedgemore}
\end{document}